\documentclass[11pt]{article}
\usepackage{amsthm}

\usepackage{graphicx}

\usepackage[english]{babel}
\usepackage{amssymb}
\usepackage{amsmath}
\usepackage{enumerate}
\usepackage{url}
\usepackage{color}
\usepackage{epsfig}
\usepackage{subfigure}
\usepackage{ifpdf}

\newcommand\inputfigeps[1]{\ifpdf
\includegraphics[scale=0.5]{#1.pdf}
\else
\includegraphics[scale = 0.5]{#1.eps}
\fi}
\newcommand\inputfigepsSmall[1]{\ifpdf
\includegraphics[scale=0.4]{#1.pdf}
\else
\includegraphics[scale = 0.4]{#1.eps}
\fi}

\DeclareMathOperator    \conv           {conv}
\DeclareMathOperator    \cone           {cone}

\DeclareMathOperator    \intr                   {int}

\DeclareMathOperator    \proj           {proj}

\DeclareMathOperator    \relint         {rel\,int}

\newcommand{\bb}{\mathbb}
\newcommand{\old}[1]{{}}

\newcommand{\R}{\bb R}

\newcommand{\Z}{\bb Z}

\newcommand\st{\mid}

\newtheorem{theorem}{Theorem}
\newtheorem{corollary}[theorem]{Corollary}
\newtheorem{lemma}[theorem]{Lemma}

\newtheorem{remark}[theorem]{Remark}

\newtheorem{claim}[theorem]{Claim}

\newtheorem{definition}[theorem]{Definition}

\newtheorem{prop}[theorem]{Proposition}

\renewcommand {\P}    {\mathcal{P}}
\newcommand   {\gp}   {\bar{g}}      
\newcommand   {\gt}    {\tilde{g}}     
\newcommand   {\gs}   {g}                

\renewcommand{\varepsilon}{\epsilon}

\usepackage[normalem]{ulem}

\begin{document}

\title{A $(k+1)$-Slope Theorem for the\\ $k$-Dimensional Infinite Group Relaxation}

\author{Amitabh Basu\thanks{Department of Mathematics, University of California, Davis,
abasu@math.ucdavis.edu}\and
Robert Hildebrand\thanks{Department of Mathematics, University of California, Davis,
rhildebrand@ucdavis.edu}\and
Matthias K\"oppe\thanks{Department of Mathematics, University of California, Davis,
mkoeppe@math.ucdavis.edu}\and
Marco Molinaro\thanks{Tepper School of Business, Carnegie Mellon University, Pittsburgh, molinaro@cmu.edu}}

\date{\today\thanks{$\relax$Revision: 143 $ - \ $Date: 2011-09-19 18:21:56 -0400 (Mon, 19 Sep 2011) $ $}}

\maketitle

\begin{abstract}
  We prove that any minimal valid function 
  for the $k$-dimensional infinite group relaxation that is piecewise linear with at
  most $k+1$ slopes and does not factor through a linear map with non-trivial
  kernel is extreme. This generalizes a theorem of Gomory and Johnson for
  $k=1$, and Cornu\'ejols and Molinaro for $k=2$.
\end{abstract}


        \section{Introduction}

                Generation of valid inequalities is a very important topic in integer programming, with a numerous literature from both theoretical and computation perspectives \cite{Bixby}. Since the structure of valid inequalities for arbitrary integer programs is too complex to be studied, researchers realized early on the necessity of understanding more restricted programs \cite{gom} (see also \cite{infinite,infinite2,alww,wolseyBook,twoStepMIR,mixingSet,mingling}). We focus on the Gomory--Johnson \emph{infinite group relaxation} \cite{infinite2}: 
                \begin{gather}
                        f + \sum_{r \in \mathbb{R}^k} r s_r \in \mathbb{Z}^k \tag{IR} \label{IR} \\
                        s_r \in \mathbb{Z}_+ \ \ \textrm{for all $r \in \mathbb{R}^k$}  \notag\\
                        s \textrm{ has finite support}. \notag
                \end{gather}
        
        The motivation behind this program is the following. For simplicity of exposition, consider a pure integer program set
        \begin{gather}
                Ay = b \tag{IP} \label{IP} \\
                y \ge 0, \ \ y \in \mathbb{Z}^d.  \notag
        \end{gather}
        Rewriting this set in tableau form with respect to a basis $B$ (i.e., pre-multiplying the system by $B^{-1}$) we obtain the equivalent system
        \begin{gather}
                y_B = \bar{b} - \bar{N} y_N \tag{IP$'$} \label{IPp} \\
                y \ge 0 \notag, \ \ y \in \mathbb{Z}^d,
        \end{gather}
        where $y_B$ are known as the basic variables, $y_N$ are called the non-basic variables, $\bar b$ is called the basic solution corresponding to $B$, and $\bar N$ are known as the non-basic columns. In \cite{gom} Gomory introduced the \emph{corner polyhedron}, which relaxes the non-negativity constraints for the basic variables $y_B$. This relaxation is commonly written as:
                \begin{gather}
                        f + \sum_{j = 1}^n r^j s_j \in \mathbb{Z}^k \tag{CP} \label{CP} \\
                        s \in \mathbb{Z}^n, \ \ s \ge 0.  \notag
                \end{gather}
        The corner polyhedron has been extensively studied in the literature specially in the restricted case $k = 1$ and some interesting results regarding its facial structure are known (see Shim and Johnson \cite{SJ}, for example). Unfortunately, the structure of this integer program still heavily relies on the specific choice of rays $r^1, \ldots, r^n$, making it difficult to analyze. 
        
         The infinite group relaxation \eqref{IR}  reduces the complexity of the system by considering all possible rays $r \in \R^k$. Therefore, \eqref{IR} is completely specified by the choice of $f$. Since this is indeed a relaxation of (IP), valid inequalities for (IR) yield valid inequalities for (IP) by restricting them to coefficients corresponding to $r^1, \ldots, r^n$. 
        
        We now briefly recall important definitions and results regarding the infinite relaxation; see \cite{corner_survey} for a more detailed discussion.
        
        \paragraph{Valid functions.} We start by defining the analog of a cut for the infinite relaxation.  We say that a function $\pi \colon \mathbb{R}^k \rightarrow \mathbb{R}$ is \emph{valid} for \eqref{IR} if $\pi \ge 0$ and the inequality
        \begin{equation}\label{eq:valid-fun}
                \sum_{r \in \mathbb{R}^k} \pi(r) s_r \ge 1
        \end{equation}
        is satisfied by every feasible solution $s$ of (IR). Note that the sum makes sense because $s$ has finite support.
        
        As pointed out in \cite{corner_survey}, the non-negativity assumption in the definition of a valid function might seem artificial at first. Although there might exist functions~$\pi$ satisfying~\eqref{eq:valid-fun} for all feasible~$s$ that take negative values~$\pi(r)$ for some~$r$, any such function must be non-negative over all \emph{rational} vectors. Since data in mixed integer linear programs are usually rational, it is natural to focus on non-negative valid functions.

        \paragraph{Minimal functions, extreme functions and facets.} Although the relaxation (IR) is less complex than the original (IP), the set of valid functions is still too broad to study.  For practical reasons, it is important that we only focus on the stronger inequalities. To make this precise, a valid function $\pi$ is said to be \emph{minimal} if there is no valid function $\pi' \neq \pi$ such that $\pi'(r) \le \pi(r)$ for all $r \in \mathbb{R}^k$. It is intuitively clear that for every valid function there is a minimal one which dominates it. However, we could not find a proof of this statement in the literature, and so we present a proof using Zorn's Lemma in Appendix~\ref{sec:minimality-proof}.

\begin{theorem}\label{thm:minimal1}
Let $\pi$ be a valid function. Then there exists a minimal valid function $\pi'$ such that $\pi' \leq \pi$.
\end{theorem}

        A function $\pi\colon \mathbb{R}^k \rightarrow \mathbb{R}$ is \emph{periodic with respect to the lattice $\Z^k$} if $\pi(r) = \pi(r + w)$ holds for all $r \in \R^k$ and $w \in \mathbb{Z}^k$. We say that  $\pi$ satisfies the \emph{symmetry condition} if $\pi(r) + \pi(-f - r) = 1$ for all $r \in \mathbb{R}^k$. Finally, $\pi$ is \emph{subadditive} if $\pi(a + b) \le \pi(a) + \pi(b)$ for all $a,b \in \R^k$.

        \begin{theorem}[Gomory and Johnson \cite{infinite}] \label{thm:minimal}
                Let $\pi \colon \mathbb{R}^k \rightarrow \mathbb{R}$ be a non-negative function. Then $\pi$ is a minimal valid function for \eqref{IR} if and only if $\pi(0) = 0$, $\pi$ is periodic with respect to $\Z^k$, subadditive and satisfies the symmetry condition.
        \end{theorem}
        
        Although minimality reduces the number of relevant valid functions that we need to study, it still leaves too many under consideration. Inspired by the importance of facets in the finite-dimensional setting, we consider the analogous concepts in this setting. A~valid function~$\pi$ is \emph{extreme} if it cannot be written as a convex combination of two other valid functions, i.e., $\pi = \frac{1}{2}\pi_1 + \frac{1}{2}\pi_2$ implies $\pi = \pi_1 = \pi_2$. For any valid function~$\pi$, let $S(\pi)$ denote the set of all $s$ satisfying~\eqref{IR} such that $\sum_{r\in \R^k}\pi(r)s_r = 1$. It is easy to verify that extreme functions are minimal. A~valid function~$\pi$ is a \emph{facet} if for every valid function~$\pi'$, we have that $S(\pi) \subseteq S(\pi')$ implies $\pi' = \pi$. This concept was introduced by Gomory and Johnson in~\cite{tspace} and we prove below that if $\pi$ is a facet, then it is extreme. Thus, facets can be seen as the strongest valid functions. 
        
\begin{lemma} If $\pi$ is a facet, then $\pi$ is extreme.
\end{lemma}
\begin{proof}
Suppose $\pi$ is a facet and let $\pi = \frac{1}{2} \pi_1 + \frac{1}{2}\pi_2$. We observe that $S(\pi) \subseteq S(\pi_1)$ and $S(\pi) \subseteq S(\pi_2)$. Let $s \in S(\pi)$. Then 
$$
1 = \sum _{r \in \R^k}\pi(r)s_r = \frac{1}{2}\sum _{r \in \R^k}\pi_1(r)s_r +\frac{1}{2}\sum _{r \in \R^k}\pi_2(r)s_r \geq \frac{1}{2} + \frac{1}{2} = 1, 
$$
so equality must hold throughout and in particular, $\sum _{r \in \R^k}\pi_i(r)s_r = 1$ for both $i=1,2$. Therefore $s\in S(\pi_i)$ for both $i=1,2$. Since $\pi$ is a facet, by definition this implies $\pi = \pi_1 = \pi_2$.
\end{proof}

        In general, constructing or even proving that a valid function is a facet or extreme can be a very difficult task. Arguably the deepest result on the infinite relaxation is a sufficient condition for facetness in the restricted setting $k = 1$, the so-called 2-Slope Theorem of Gomory and Johnson \cite{infinite2, tspace}. 
        
        \begin{theorem}[Gomory--Johnson 2-Slope Theorem]
        \label{thm:2-slope}
                Let $\pi \colon \mathbb{R} \rightarrow \mathbb{R}$ be a minimal valid function. If $\pi$ is a continuous piecewise linear function with only two slopes, then $\pi$ is a facet (and hence extreme).
        \end{theorem}
        
        In addition to its theoretical appeal, this result also has practical relevance. It supplies theoretical indication about the intrinsic power of 2-slope functions, which are very effective cuts in integer programming solvers \cite{Bixby} (e.g., GMI's). This surprising result was already known in the 1970s, and despite the increased efforts in understanding relaxations for \eqref{IP} with $k > 1$, a generalization of this result was obtained only recently for the case $k = 2$     by Cornu\'ejols and Molinaro \cite{3slope}.
        
        \begin{theorem}[3-Slope Theorem]\label{thm:3-slope}
                Let $\pi \colon \mathbb{R}^2 \rightarrow \mathbb{R}$ be a minimal valid function. If $\pi$ is a continuous piecewise linear function with only 3~slopes and with 3~directions, then $\pi$ is extreme. 
        \end{theorem}   
        Here, the \emph{directions} of the function refer to the direction of the edges bounding the cells of the piecewise linear function.  Assuming 3~directions ensures that we do not include 3-slope functions that are constant in some direction. See Figure \ref{fig:partition} for an illustration.
The authors show in~\cite{3slope} that this assumption cannot be removed.

Theorems~\ref{thm:2-slope} and \ref{thm:3-slope} contribute a simple sufficient condition for extremality that avoids the repetition of (often long) arguments for specific families of cuts. For instance, the valid functions that are proved to be extreme in \cite{deyRichard} can all be shown to satisfy the hypotheses of Theorem~\ref{thm:3-slope}; hence Theorem~\ref{thm:3-slope} is a powerful unification of all those arguments. Our goal in this paper is to prove such a theorem for general $k$.

\begin{figure}
\label{fig:partition}
\begin{center}
$$
\begin{array}{rcl}
\inputfigepsSmall{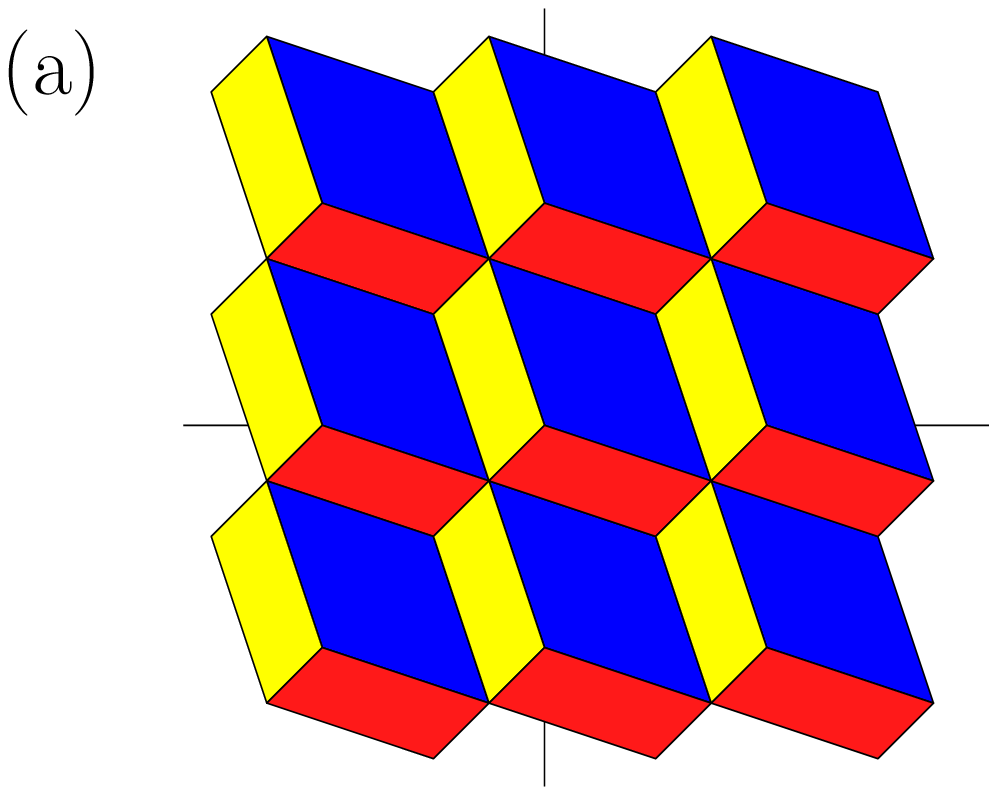}&\ \ \ \ \ \ &
\inputfigeps{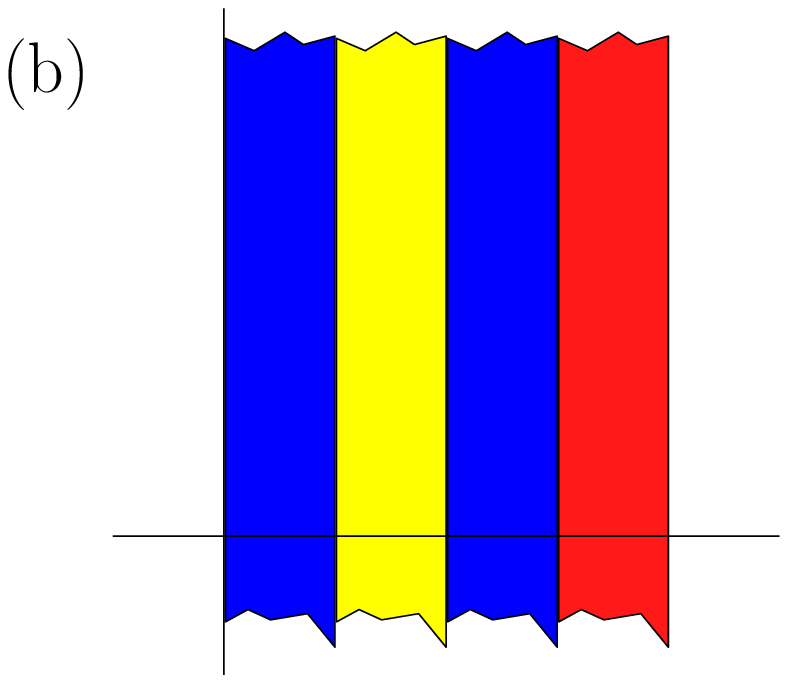}
\end{array}
$$
\end{center}
\caption{The assumptions of the Cornu\'ejols--Molinaro 3-slope theorem (Theorem \ref{thm:3-slope}).  (a)  A function with 3 slopes and 3 directions, where each color represents the cells where the function has a certain slope.  (b) A function with 3 slopes and only 
1~direction.
}
\end{figure}

        \paragraph{Our results.} We generalize the above results and present a sufficient condition for facetness (and therefore for extremality) of valid functions for arbitrary dimension $k$.  For this, we  generalize the notion of a function having 3 directions.

\begin{definition}
A function $\theta\colon \R^k \rightarrow \R$ is {\em genuinely $k$-dimensional}
if there does not exist a function $\varphi \colon \R^{k-1} \rightarrow \R$ and a linear
map $T \colon \R^k \rightarrow \R^{k-1}$ such that $\theta = \varphi\circ T$. 
\end{definition}

In the context of piecewise linear functions~$\pi\colon\R^2\to\R$, the above definition is related to the concept of functions with 3 directions as used in \cite{3slope}. Indeed, Appendix~\ref{sec:generalizes} shows that the assumption of being genuinely $2$-dimensional is a weaker assumption; namely, if a function has 3~directions, then it is genuinely $2$-dimensional.

\begin{theorem}\label{thm:main}
Let $\pi\colon \R^k \to \R$ be a minimal valid function that is piecewise linear with a locally
finite cell complex and genuinely $k$-dimensional with at most $k+1$ slopes. Then $\pi$ is a facet (and therefore extreme) and has exactly $k+1$ slopes.
\end{theorem}   

        This settles an open question posed by Gomory and Johnson in \cite{tspace}. 

One direct application of this result is to investigate the facetness of certain valid functions studied in \cite{basu-cornuejols-koeppe:unique-minimal-liftings-simplicial}. There is a useful procedure known as the \emph{trivial lifting procedure} which can be used to derive minimal valid functions for \eqref{IR} using the Minkowski functionals of maximal lattice-free convex sets. This procedure was studied in \cite{basu-cornuejols-koeppe:unique-minimal-liftings-simplicial} as applied to maximal lattice-free simplices. It turns out that for a special class of such simplices, the minimal valid functions obtained will have $k+1$ slopes and Theorem~\ref{thm:main} can be directly applied to prove that they are facets.
                
        The high-level structure of the proof of Theorem \ref{thm:main} is similar to the proof of the 2-Slope and 3-Slope Theorems presented in \cite{tspace} and \cite{3slope}. Let $\pi \colon \mathbb{R}^k \rightarrow \mathbb{R}$ be a valid function satisfying the assumptions of the theorem. We consider an arbitrary valid function~$\pi'$ such that $S(\pi) \subseteq S(\pi')$, and our goal will be to show that $\pi =\pi'$. In order to achieve this, we first prove in Section \ref{sec:compatible} that if $\pi$ is affine in a region of $\mathbb{R}^k$, then $\pi'$ is also affine in this region, albeit with a different gradient. 

The next step, Section \ref{sec:gradient-set}, is to write a system of equations which is satisfied by the gradients of $\pi$ as well as the gradients of $\pi'$. This is the most involved step in the proof. Some properties proved in \cite{3slope} rely on arguments about low dimensional objects; whereas, the appropriate high dimensional generalizations require more sophisticated techniques. In particular, we make use of a topological lemma about closed coverings of the simplex, the Knaster--Kuratowski--Mazurkiewicz Lemma (Lemma \ref{lem:KKM}), which in the one-dimensional case reduces to the easy fact that an interval cannot be covered by two disjoint closed sets. 

The final step, Section \ref{sec:unique}, is then to prove that this system has a unique solution, which  implies that the gradients of $\pi$ and $\pi'$ are the same.  This,  together with the fact that $\pi(0) = \pi'(0) = 0$,  implies that $\pi = \pi'$. The proof of this last step simplifies the one presented in \cite{3slope} and directly exposes the properties driving the uniqueness of the system.


\section{Preliminaries}

\subsection{Basic polyhedral theory}

        In this section we collect some basic definitions from polyhedral theory (see~\cite{z}) as well as two simple lemmas that will be used throughout the text.  The open ball of radius $\varepsilon$ around a point $r$ will be denoted by $B_\varepsilon(r)$.

\begin{definition}
A {\em polyhedral complex} in $\R^k$ is a collection $\P$ of polyhedra in $\R^k$ such that:
\begin{enumerate}[\rm(i)]
\item if $P \in \P$, then all faces of $P$ are in $\P$,
\item the intersection $P \cap Q$ of two polyhedra $P,Q \in \P$ is a face of both $P$ and $Q$.
\end{enumerate}
\end{definition}

A cell $P$ is {\em maximal} if there is no $Q \in\P$ containing it. The complex is {\em pure} if all
maximal cells have the same dimension. The complex is said to be {\em complete} if the union of all elements in the complex is $\R^k$. A {\em subcomplex} of $\P$ is a subset $\P' \subseteq \P$ such that $\P'$ is itself a polyhedral complex.
 
A {\em polyhedral fan} is a polyhedral complex of a finite number of cones. Given a polyhedral fan $\mathcal{F}$ and $r\in \R^k$, we use the notation $\mathcal{F} + r = \{\, C + r \st C \in \mathcal{F}\,\}$, which is a finite polyhedral complex. Given a cone $C \subseteq \R^k$, a {\em triangulation} of $C$ is a polyhedral fan $\P$ such that each element of $\P$ is a simplicial cone and the union of all elements in $\P$ is $C$. Given a polyhedral fan $\P$, a {\em triangulation} of $\P$ is a polyhedral fan $\mathcal{F}$ such that for every element $P \in\P$, there exists a triangulation of $P$ as a subcomplex of $\mathcal{F}$ and every element of $\mathcal{F}$ is simplicial.

Given a polyhedral complex $\P$ and a set $X \subseteq \R^k$, we use the
notation $\P \cap X$ to denote the collection of sets $\{\, P \cap X \st P\in \P
\textrm{ and } P\cap X \neq \emptyset\,\}$. Observe that if $\P$ is complete,
the union of all sets in $\P \cap X$ is $X$. 

With a slight abuse of notation, for any point $v \in \R^k$ and a polyhedral complex $\P$, we will use $v \in \P$ to denote that $v \in P$ for some element $P \in \P$.

\begin{definition}\label{assum:fan}
A polyhedral complex $\P$ is called \emph{locally finite} if for every point~$r\in\R^k$ there exists an open ball $B_\varepsilon(r)$ around~$r$, such that $\P\cap B_\varepsilon(r)$ equals $(\mathcal{F}_r + r) \cap B_\varepsilon(r)$ for some polyhedral fan $\mathcal{F}_r$ (recall that a polyhedral fan is finite by definition).
\end{definition}

Notice that the above definition is equivalent to stating that each point in $\R^k$ has a neighborhood which intersects only finitely many elements of $\P$. In addition, using standard arguments, it is easy to see that this finite intersection extends from points to compact sets. 

\begin{prop} \label{prop:complex-finite-on-compact}
  Let $\P$ be a locally finite polyhedral complex. Then for every compact set $K \subseteq \R^k$, only finitely many elements of $\P$ intersect $K$.
\end{prop}


        Now we present two simple linear algebraic facts that will be useful for the next sections. 

\begin{lemma}\label{lem:lin-indep}
If $r^1, \ldots, r^{k+1} \in \R^k$ are such that $\cone(r^i)_{i=1}^{k+1} =
\R^k,$ then every proper subset of $\{r^1, \ldots, r^{k+1}\}$ is linearly independent. 
\end{lemma}

\begin{proof}
It suffices to show that every $k$-subset of $\{r^1, \ldots, r^{k+1}\}$ is linearly independent. 
Without loss of generality, we will just show that $r^1, \ldots, r^k$ are linearly independent. 
If not, then there exists a hyperplane $H$ containing the linear span of $r^1, \ldots, r^k$. Suppose $H_+$ is the half-space defined by $H$ containing $r^{k+1}$. 
This implies that $\cone(r^i)_{i=1}^{k+1} \subseteq H_+$, contradicting the fact that  $\cone(r^i)_{i=1}^{k+1} = \R^k$.
\end{proof}

\begin{lemma}\label{lem:cone-span}
Let $\{a^1, \ldots, a^{k+1}\}$ and $\{b^1, \ldots, b^{k+1}\}$ be two sets of $k+1$
vectors in $\R^k$. Suppose $\cone(a^i)_{i=1}^{k+1} = \R^{k}$, and $a^i\cdot
b^j < 0$ for $i\neq j$. Then $\cone(b^j)_{j=1}^{k+1} = \R^k$. 
\end{lemma}
\begin{proof}
We show that the cone $X = \cone(b^j)_{j=1}^{k+1} = \R^k$ by considering the polar cone
$$X^\circ = \{\, r \in \R^k\st b^j\cdot r \leq 0, \;\;j = 1, \ldots, k+1\, \} $$ and equivalently showing that $X^\circ = \{0\}$. 
Consider any vector $r^0 \neq 0$. Since  $\cone(a^i)_{i=1}^{k+1} = \R^{k}$, by
Carath\'eodory's theorem, there exists $j\in\{1, \ldots, k+1\}$ such that $-r^0
= \sum_{i \neq j}\lambda_i a^i$ with $\lambda_i \geq 0$. Since $b^j\cdot a^i <
0$ for all $i \neq j$, we have that $b^j\cdot(-r^0) < 0$, or equivalently, $b^j\cdot r^0 >
0$. Thus $r^0 \not\in X^\circ$.
\end{proof}

\subsection{Piecewise linear functions} 

We now give a precise definition of piecewise linear functions and
related notions.
        
\begin{definition}
  Let $\mathcal{P}$ be a pure, complete polyhedral complex in $\R^k$ and let
  $\{\mathcal{P}_i\}_{i \in I}$ be a partition of the set of maximal cells of
  $\mathcal{P}$.  Consider a function $\theta \colon \R^k \rightarrow \R$
  such that for each $i\in I$, there exists a vector $g^i \in \R^k$ such
  that for every $P \in \mathcal{P}_i$, there exists a constant $\delta_P$
  such that $\theta(r) = g^i \cdot r + \delta_P$ for all $r \in P$.
  Then $\theta$ is called a \emph{piecewise linear function}, more specifically
  a \emph{piecewise linear function with cell complex~$\mathcal P$}, 
  and a \emph{piecewise linear function compatible with $\{\mathcal{P}_i\}_{i \in I}^{}$}.
\end{definition}

Given a piecewise linear function $\theta$ with cell complex $\P$, for any maximal cell $P \in \P$ let $g^P \in \R^k$ denote the vector such that $\theta(r) = g^P\cdot r +\delta_P$ for some constant $\delta_P$. 
We define the equivalence relation $\sim$ on the maximal elements of $\P$ according to their gradients as $P \sim P'$ if and only if $g^P = g^{P'}$. 
Each equivalence class defines a subcomplex $\P_i \subseteq \P$, $i\in I$ for some index set $I$. We say that $\theta$ has {\em n slopes} if $I$ is finite and $|I|=n$. The {\em gradient set} of a piecewise linear function is the set of vectors $\{g^i\}_{i\in I}$ corresponding to the equivalence classes $\P_i$, namely $g^i = g^P$ for all $P \in \P_i$.


\subsection{Lipschitz continuity}

The following two lemmas assert strong continuity properties of piecewise linear, subadditive  functions with a locally finite cell complex.  We will use $\| \cdot\|$ to denote the Euclidean norm.

\begin{lemma}\label{lem:PWL+LF=LLC}
Let $\theta\colon\R^k \rightarrow \R$ be a piecewise linear function with a 
locally finite cell complex.
Moreover, suppose $\theta(0) = 0$. Then $\theta$ is locally Lipschitz
continuous at the origin, i.e., there exist $\varepsilon > 0$ and $K > 0$ such
that $|\theta(r)| \leq K \|r\|$ for all $r \in B_\varepsilon(0)$. 
\end{lemma}

\begin{proof}
Let $\theta$ be a piecewise linear function with a locally finite cell complex $\P$. Since $\P$ is locally finite, there exists an open ball~$B_\varepsilon(0)$ around the origin such that $\P\cap B_\varepsilon(0) =
\mathcal{F}\cap B_\varepsilon(0)$ for some complete polyhedral fan $\mathcal{F}$. This
implies that there exists a \emph{finite} subcomplex $\P' \subseteq \P$ such
that $\P\cap B_{\varepsilon}(0) = \P' \cap B_{\varepsilon}(0)$ and every maximal element of $\P'$ contains the
origin. Therefore, the union of all $P$ in $\P'$ contains $B_{\varepsilon}(0)$.

Since $\theta$ is piecewise linear and $\theta(0)=0$, for every maximal element $P\in \P'$, there exists $g^P\in \R^k$ such that $\theta(r) = g^P\cdot r$ for all $r \in P$ and so $|\theta(r)| \leq \|g^P\|\|r\|$ by the Cauchy--Schwarz inequality. Let $K = \max\{\,\|g^P\|\st P \in \P' \,\}$. Then  $| \theta(r)|\leq K\| r\|$ for all $r\in B_\varepsilon(0)$.
\end{proof}

\begin{lemma}\label{lem:LLC+SA=LC}
Let $\theta\colon\R^k \rightarrow \R$ be a subadditive function that is locally Lipschitz continuous at the origin. Then $\theta$ is (globally) Lipschitz continuous.
\end{lemma}

\begin{proof}
We first show that there exist $K > 0$ and $\varepsilon > 0$, such that given any $\tilde r\in \R^n$, $| \theta(r)-\theta(\tilde r)|\leq K\| r - \tilde r\|$ for all $r\in\R^n$ satisfying $\|r-\tilde r\|<\varepsilon$. Indeed, since $\theta$ is locally Lipschitz continuous at the origin, it follows that there exists $K > 0$, $\varepsilon>0$ such that, for all $r\in\R^n$ satisfying $\|r-\tilde r\|<\varepsilon$, we have $|\theta(r-\tilde r)|\leq K\| r - \tilde r\|$.
Hence, for all $r\in\R^n$ satisfying $\|r-\tilde r\|<\varepsilon$, $$|\theta(r)-\theta(\tilde r)|\leq \max\{\theta(\tilde r-r),\theta(r-\tilde r)\}\leq K\| r - \tilde r\|,$$
where the first inequality follows from the subadditivity of $\theta$.

We now show that for any $\tilde r, r \in \R^k$, $|\theta(\tilde r)-\theta(r)|
\leq K\| \tilde r - r\|$. Define $m$ to be an integer larger than $\| \tilde r - r\|/\varepsilon$. Let $r^i = \frac{i}{m}(r - \tilde r) + \tilde r$, so $r^0 = \tilde r$ and $r^m = r$. Moreover, $\|r^{i+1} - r^i\| = \|\tilde r - r\|/m < \varepsilon$. Therefore, $|\theta(r^{i+1})-\theta(r^i)| \leq K\|r^{i+1} - r^i\|$ for all $i = 0, \ldots, m-1$. Hence,
\begin{equation}
|\theta(\tilde r) - \theta(r)| \leq \sum_{i=0}^{m-1}|\theta(r^{i+1}) -
\theta(r^i)| \leq \sum_{i=0}^{m-1}K\|r^{i+1} - r^i\| = K\|\tilde r - r\|. \pushQED{\relax}\tag*{\qed}
\end{equation}
\end{proof}

\subsection{Properties of genuinely $k$-dimensional functions}
Now we focus on properties that we gain by imposing that a function is genuinely $k$-dimensional. We will need the following lemma, which is implied by Lemma 13 in~\cite{bccz} and is a consequence of Dirichlet's Approximation Theorem for the reals.

\begin{lemma}\label{lem:half-line}
Let $y\in \R^k$ be any point and $r \in \R^k\setminus \{0\}$ be any direction. Then for every $\varepsilon > 0$ and $\bar\lambda \geq 0$, there exists $w \in \Z^k$ such that $y + w$ is at distance less than $\varepsilon$ from the half line $\{y + \lambda r \st \lambda \geq \bar\lambda\}$.
\end{lemma}

\begin{lemma}\label{lem:genuinely-k-dim}
Let $\theta \colon \R^k \rightarrow \R$ be non-negative, Lipschitz continuous, subadditive and periodic with respect to the lattice $\Z^k$. 
Suppose there exist $r \in \R^k\setminus\{0\}$ and $\bar\lambda >0$ such that
$\theta(\lambda r) = 0$ for all $0\leq \lambda \leq \bar\lambda$. Then $\theta$ is not genuinely $k$-dimensional.
\end{lemma}

\begin{proof}
Let the Lipschitz constant for $\theta$ be $K$, that is, $|\theta(x)-\theta(y)| \leq K\|x-y\|$ for all $x,y\in \R^k$.

We will begin by showing that $\theta(\lambda r) = 0$ for all $\lambda \in \R$.  Let $\lambda' \in \R$.

Suppose that $\lambda' > \bar \lambda$ and let $M \in \Z_+$ such that $0 \leq \lambda'/M \leq \bar \lambda$.  From the hypothesis, we have that $\theta(\frac{\lambda'}{M} r)= 0$.  By non-negativity and subadditivity of $\theta$ we see
$
0 \leq \theta(\lambda' r) \leq M \theta(\frac{\lambda'}{M} r)= 0,
$
and therefore, $\theta(\lambda' r) = 0$. This shows that $\theta(\lambda r) = 0$ for all $\lambda \geq 0$.

Next suppose $\lambda' < 0$.  By Lemma~\ref{lem:half-line}, for all $\varepsilon >0$ there exists a $w \in \Z^k$ such that $\lambda' r + w$ is at distance less than $\varepsilon$ from the half line $\{\lambda' r + \lambda r\st \lambda \geq -\lambda'\} = \{\lambda r \st \lambda \geq 0\}$.  That is, there exists a $\tilde\lambda \geq 0$ such that $\|\lambda' r + w - \tilde \lambda r\| \leq \varepsilon$. Since $\theta(\tilde \lambda r) = 0$, by periodicity and then Lipschitz continuity, we see that
$
0 \leq \theta(\lambda' r) = \theta(\lambda'r + w) = \theta( \lambda' r + w) - \theta(\tilde \lambda r) \leq K \varepsilon.
$
This holds for every $\varepsilon >0$ and therefore $\theta(\lambda'r) = 0$.  Thus, we have shown that $\theta(\lambda r) = 0$ for all $\lambda \in \R$.

Let $L = \{ \lambda r \st \lambda \in \R\}$.  We claim that if $x-y \in L$, then $\theta(x) = \theta(y)$. Since $x-y \in L$, as shown above, $\theta(x-y) = 0$.  By subadditivity, $\theta(y) + \theta(x-y) \geq \theta(x)$, which implies $\theta(y) \geq \theta(x)$. Similarly,  $\theta(x) \geq \theta(y)$, and hence we have equality.

We conclude that $\theta = \varphi \circ \proj_{L^\perp}$ for some function $\varphi \colon \R^{k-1} \rightarrow \R$ and therefore $\theta$ is not  genuinely $k$-dimensional.
\end{proof}

\begin{lemma}\label{lem:gi-cone}
Let $\theta \colon \R^k \rightarrow \R$ be non-negative, piecewise linear with
a locally finite cell complex, subadditive, periodic with respect to the lattice $\Z^k$ and genuinely $k$-dimensional with at most $k+1$ slopes and suppose $\theta(0) = 0$.  Then $\theta$ has exactly $k+1$ slopes. Let the gradient set of $\theta$ be the vectors $g^1, \ldots, g^{k+1}$; then they satisfy $\cone(g^i)_{i=1}^{k+1} = \R^k$.  Furthermore, for every $i=1, \ldots, k+1$, there exists a maximal cell $P \in \P_i$ such that $0 \in P$.
\end{lemma}

\begin{proof}
First we note that since $\theta$ is  subadditive, satisfies $\theta(0)=0$ and is piecewise linear with a locally finite cell complex, Lemmas~\ref{lem:PWL+LF=LLC} and~\ref{lem:LLC+SA=LC} imply that $\theta$ is a Lipschitz continuous function.

We label the gradient set of $\theta$ as $g^1, \ldots, g^n$ and the corresponding subcomplexes as $\P_1, \ldots, \P_n$ where $n \leq k+1$. Without loss of generality, assume that $0 \in \P_i$ for $i \leq m$ and $0 \not\in \P_i$ for $i > m$ for some $m \leq n \leq k+1$. Let $C = \{\,r \in \R^k \st
g^i\cdot r \leq 0, \;\; i = 1, \ldots, m\,\}$ be the polar cone of
$\cone(g^i)_{i=1}^m$. We show that $C = \{0\}$, which implies that
$\cone(g^i)_{i=1}^m = \R^k$. This would imply that $m = k+1$ and
$\cone(g^i)_{i=1}^{k+1} = \R^k$. Moreover, this would imply that $0 \in \P_i$ for every $i$ and so there exists a maximal cell in $\P_i$ containing $0$.

Suppose there exists $r^0\in C\setminus \{0\}$. Since $\theta$ has a locally finite
cell complex, there exists an open ball $B_\varepsilon(0)$  such that $\mathcal
P\cap B_\varepsilon(0) = \mathcal F\cap B_\varepsilon(0)$, where $\mathcal F$ is a polyhedral fan where every maximal cell contains the origin. 
Since $0 \in \P_i$ for $i\leq m$ and $0 \not\in \P_i$ for $i > m$, there exists $0 < \delta < \varepsilon$ such that $B_\delta(0)$ intersects only $\P_1, \ldots, \P_m$. 
Let $\bar \lambda > 0$ such that $\lambda r^0  \in B_\delta(0)$ for
all $0\leq \lambda \leq \bar\lambda$. Since $r^0  \in C$, we see that $g^i\cdot r^0  \leq 0$ for all
$i=1, \ldots, m$.  
Since $\mathcal F$ is a polyhedral fan, the line segment
from $0$ to $\bar \lambda r^0 $ lies completely within a cell~$P' \in \P_i$ for
some~$i=1,\dots,m$.  Thus $0 \leq \theta(\lambda r^0 ) = \lambda g^i\cdot r^0  \leq 0$ for
all $0 \leq \lambda \leq \bar\lambda$. 
But then by Lemma~\ref{lem:genuinely-k-dim}, $\theta$ is
not genuinely $k$-dimensional. This is a contradiction. 
\end{proof}


\subsection{Line integrals}

The following discussion shows that we can compute line integrals of the
gradients of $(k + 1)$-slope functions. 
We choose to restrict ourselves to functions
with locally finite cell complexes.  This is motivated by the necessity of
excluding certain pathological cases where the following does not hold and
allows us to give a completely elementary proof.  We remark that this
restriction precludes handling some important functions such as the ones
constructed in \cite{bccz08222222}, which do not have locally finite cell
complexes.  More general versions of Lemma~\ref{lemma:integration} below can,
of course, be proved using the Lebesgue version of the fundamental theorem of
calculus.

        \begin{lemma} \label{lemma:integration}
                Consider a locally finite, complete polyhedral complex $\mathcal{P}$ in $\R^k$ and let $\{\mathcal{P}_i\}_{i = 1}^{k + 1}$ be a partition of the set of maximal cells of $\mathcal{P}$. Fix a point $r \in \R^k$. Then there exist $\mu_1, \mu_2, \ldots, \mu_{k+1} \in \R_+$ with $\sum_{i = 1}^{k + 1} \mu_i = 1$ such that for every function $\theta$ that is piecewise linear compatible with $\{\mathcal{P}_i\}_{i = 1}^{k + 1}$ with gradients $g^1, \ldots, g^{k+1}$ corresponding to this partition, the following holds. 
                 \begin{align*}
                        \theta(r) = \theta(0)+ \sum_{i = 1}^{k + 1} \mu_i (g^i \cdot r).
                 \end{align*}
        \end{lemma}
        
        \begin{proof}
        
        Let $\rho \colon [0,1] \rightarrow \R^k$ be the parameterization of the segment $[0, r]$ given by $\rho(\lambda) = \lambda r$.         
        Let $\mathcal{Q}_i = \{\,\rho^{-1}(P \cap [0,r]) \st P \in \mathcal{P}_i\,\}$. By convexity, $\mathcal{Q}_i$ is a family of intervals in $[0,1]$ (some of these intervals could be degenerate). Moreover, since $\mathcal{P}$ is locally finite, Remark~\ref{prop:complex-finite-on-compact} guarantees that $\mathcal{Q}_i$ is a finite family. In addition, since $\mathcal{P}$ is complete, the union of the intervals in $\bigcup_{i = 1}^{k+1} \mathcal{Q}_i$ equals $[0,1]$. 
        
        Using the finiteness of the above families, let $0 = \lambda_0 \leq \lambda_1 \leq \ldots \leq \lambda_n = 1$ be the end-points of the intervals in $\bigcup_{i = 1}^{k + 1} \mathcal{Q}_i$; i.e., each interval $[\lambda_j, \lambda_{j+1}]$ is an interval in $\mathcal{Q}_i$, for some $i = 1, \ldots, k+1$. This implies that $\rho([\lambda_j, \lambda_{j+1}])$ is contained in a polyhedron in $\mathcal{P}_i$. In this case, the compatibility of $\theta$ with $\{\P_i\}_{i=1}^{k+1}$ gives $$\theta(\lambda_{j+1} r) - \theta(\lambda_j r) = g^i\cdot (\lambda_{j+1}r) - g^i\cdot (\lambda_j r) = (\lambda_{j+1} - \lambda_j)(g^i\cdot r).$$ Therefore,
\begin{equation}\label{eq:function-value}\theta(r) - \theta(0) = \sum_{j=0}^{n-1}(\theta(\lambda_{j+1} r) - \theta(\lambda_j r))  = \sum_{i=1}^{k + 1}|\mathcal{Q}_i|(g^i\cdot r),
\end{equation} 
where $|\mathcal{Q}_i|$ is the sum of the lengths of all the intervals in $\mathcal{Q}_i$. Setting $\mu_i = |\mathcal{Q}_i|$\ completes the result.       
        \end{proof}


\section{Proof of Theorem~\ref{thm:main}}\label{sec:main-proof}

We now concentrate on a function $\pi$ which satisfies the hypothesis, i.e., $\pi$ is a minimal valid
function that is piecewise linear with a locally finite cell complex and
genuinely $k$-dimensional with at most $k+1$ slopes. We recapitulate properties of $\pi$ that we have learned.  
Theorem~\ref{thm:minimal} shows that $\pi(0) = 0$, $\pi$ is subadditive, periodic with respect to the lattice $\Z^k$, and satisfies the symmetry condition.  Lemma~\ref{lem:PWL+LF=LLC} and Lemma~\ref{lem:LLC+SA=LC} show that $\pi$ is Lipschitz continuous. Lemma~\ref{lem:gi-cone} uses the assumption of $\pi$ being a genuinely $k$-dimensional function and shows that $\pi$ has exactly $k+1$ slopes.  Let $\pi$ be a piecewise linear function with cell complex $\P$ and we denote the gradient set of $\pi$ by $\{\gp^1, \ldots, \gp^{k+1}\}$ and the subcomplex corresponding to vector $\gp^i$ as $\P_i$. Lemma \ref{lem:gi-cone} also shows that $0 \in \P_i$ for all $i =1, \ldots, k+1$.

The proof structure is guided by the so-called {\em Facet Theorem} proved in~\cite{infinite}. For the sake of completeness and because of differences in notation, we provide a statement of this theorem below and a self-contained proof in Appendix~\ref{sec:facet}. For any valid function $\theta$, let $E(\theta)$ denote the set of all pairs $(u, v) \in \R^k\times \R^k$ such that $\theta(u+v) = \theta(u) + \theta(v)$.

\begin{theorem}[Facet Theorem] \label{thm:facet-theorem}  Let $\pi$ be a minimal valid function. Suppose that for every minimal valid function $\pi'$, we have that $ E(\pi) \subseteq E(\pi')$ implies $\pi' = \pi$. Then $\pi$ is a facet.
\end{theorem}

\paragraph{Main Goal. }We consider any minimal valid function $\pi'$ such that $E(\pi) \subseteq E(\pi')$ and show that $\pi' = \pi$. By Theorem~\ref{thm:facet-theorem}, this will imply that $\pi$ is a facet. 

Since $\pi'$ is minimal, by Theorem~\ref{thm:minimal}, $\pi'$ is non-negative, subadditive, periodic with respect to the lattice $\Z^k$. Moreover, $\pi'(0)=0$ and the symmetry condition holds, i.e., $\pi'(r) + \pi'(-f-r) = 1$ for all $r \in \R^k$, and because of periodicity, $\pi'(w-f) = 1$ for every $w \in \Z^k$. Finally, the symmetry condition and non-negativity of $\pi'$ implies that $\pi'$ is bounded above by $1$. \smallskip

For the following proof, we will use $\theta$ when we wish to refer to  a more general function than $\pi'$.

\subsection{Compatibility} \label{sec:compatible}

 We show that $\pi'$ is a piecewise linear function compatible with $\{\mathcal{P}_i\}_{i = 1}^{k + 1}$. 

The idea of the proof is the following. First, using the partial additivity of $\pi'$ implied by $E(\pi) \subseteq E(\pi')$, we show that $\pi'$ is affine in parallelotopes around the origin. Then, we use translates of these parallelotopes to show that $\pi'$ is affine in each maximal cell of~$\P$. In fact, our arguments will imply that $\pi'$ has the same gradient in every maximal cell in~$\P_i$, which then gives the desires result. 

In order to carry out the first step, we need the following lemma, which appears as Lemma~5.8 in \cite{corner_survey} and was proved in~\cite{bccz08222222}.

\begin{lemma}[Interval Lemma] \label{lem:interval_lemma} Let $\theta \colon \R \rightarrow
\mathbb{R}$ be a function bounded on every bounded interval. Given real numbers $u_1 < u_2$ and $v_1 < v_2$, let $U = [u_1, u_2]$, $V =
[v_1, v_2]$, and $U + V = [u_1 + v_1, u_2 + v_2]$.
If $\theta(u)+\theta(v) = \theta(u+v)$ for every
$u \in U$ and $v \in V$, then there exists $c\in \R$ such that
\begin{alignat*}{2}
  \theta(u)&=\theta(u_1)+c(u-u_1) &\quad& \text{for every $u\in U$,} \\
  \theta(v)&=\theta(v_1)+c(v-v_1) && \text{for every $v\in V$,} \\
  \theta(w)&=\theta(u_1+v_1)+c(w-u_1-v_1) && \text{for every $w\in U+V$.}
\end{alignat*}
\end{lemma}

\begin{lemma}[$\pi'$ is linear on parallelotopes at the origin]\label{lemma:simplex_S} Let $P_0$ be a cell in $\P$ containing the
  origin. Consider any parallelotope $\Pi \subset P_0$ such that: {\rm(i)} $0 \in \Pi$
  and {\rm(ii)} $\Pi + \Pi \subseteq P_0$. Then there exists $g'$ such that $\pi'(r) =
  g'\cdot r$ for all $r \in \Pi$.
\end{lemma} 

\begin{proof}

        Since $\Pi$ contains the origin, let $v^1, \ldots, v^n$ be generating vectors of $\Pi$, namely these are linearly independent vectors such that $\Pi = \{\,\sum_{i = 1}^n \lambda_i v^i \st \lambda_i \in [0,1] \   \text{ for } i=1, \ldots, n\,\}$. 
        
        We first claim that for all $r,r' \in \Pi$, we have that $\pi'(r) + \pi'(r') = \pi'(r + r')$. 
        To see this, recall that $\pi$ is affine in $P_0$, and hence in $\Pi$. Since $\pi(0) = 0$, $\pi$ is actually linear in $\Pi$. 
        Using the fact that $\Pi + \Pi \subseteq P_0$, we obtain that $\pi(r) + \pi(r') = \pi(r + r')$ for all $r,r' \in \Pi$; 
        since $E(\pi) \subseteq E(\pi')$, the same holds for $\pi'$, which proves the claim.
        
Fix any $i \in \{1, \ldots, n\}$. We claim that $\pi'$ is linear on the segment
$[0,v^i]$. Consider the function $\phi(\lambda) = \pi'(\lambda
v^i)$, which by the previous paragraph is additive over $[0,1]$, i.e.,
$\phi(\lambda) + \phi(\lambda') = \phi(\lambda +
\lambda')$ for all $\lambda, \lambda' \in [0,1]$. Since $\pi'$ (and hence $\phi$) is bounded,  the Interval Lemma
(Lemma~\ref{lem:interval_lemma}) applied to $\theta$ implies there exists a scalar $\alpha_i$ such that $\pi'(\lambda v^i) = \phi(\lambda) = \alpha_i \lambda + \phi(0)$ for all $\lambda \in [0,1]$. Since $\pi'(0)=0$, we also have $\phi(0)   = 0$. Therefore, $\pi'(\lambda v^i) = \alpha_i \lambda$.
        
        As $v^1, \ldots, v^n$ are linearly independent, there exists a $g' \in \mathbb{R}^k$ such that $g' \cdot v^i = \alpha_i$ for all $i =1, \ldots, n$. 
        We claim that $\pi'(r) = g' \cdot r$ for all $r \in \Pi$. By letting $r = \sum_{i = 1}^n \lambda_i v^i$, the result follows because $\pi'$ is additive on $\Pi$. 
        \begin{equation}
        \pi'(r) = \pi'\Big(\sum_{i = 1}^n \lambda_i v^i\Big) 
        = \sum_{i = 1}^n \pi'(\lambda_i v^i) 
        = \sum_{i = 1}^n \alpha_i \lambda_i 
        = g'\cdot \Big( \sum_{i = 1}^n \lambda_i v^i \Big) 
        = g'\cdot r.\pushQED{\relax}\tag*{\qed}
      \end{equation}
\end{proof}

Before we proceed, we prove a technical lemma about the continuity of $\pi'$. The motivation is  to prove that $\pi'$ is affine in each maximal cell of $\P$ by showing this for the interior of each cell, and then the full result follows by continuity.

\begin{lemma}\label{lemma:theta-lipschitz}
$\pi'$ is Lipschitz continuous.
\end{lemma}

\begin{proof}
We first show that $\pi'$ is locally Lipschitz continuous at $0$, i.e., there exist $K > 0$, $\varepsilon > 0$ such that $| \pi'(r)|\leq K\| r\|$
for all $r\in B_{\varepsilon}(0)$.  Since $\pi$ has a locally finite cell
complex, there exists a neighborhood $B_{\varepsilon_1}(0)$ of the origin satisfies $\P\cap B_{\varepsilon_1}(0)= \mathcal{F}\cap B_{\varepsilon_1}(0)$ for some complete polyhedral fan $\mathcal{F}$. 
We consider a triangulation $\bar{\mathcal{F}}$ of $\mathcal{F}$, i.e., $\bar{\mathcal{F}}$ contains a triangulation for each cone in $\mathcal{F}$ and every element of
$\bar{\mathcal{F}}$ is simplicial. Consider any maximal simplicial cone $C\in \bar{\mathcal{F}}$ and consider $P\in \P$ such that
$C\cap B_{\varepsilon_1}(0) \subseteq P$ (note that such a $P$ exists because $\bar{\mathcal F}$ is a triangulation of $\mathcal F$). Then there exist generators $\{v^1_C,
\ldots, v^{k}_C\}$ for $C$ such that the parallelotope $\Pi$ formed by $\{v^1_C,
\ldots, v^{k}_C\}$ is such that $\Pi+\Pi \subseteq P$. We do this construction for
all maximal elements of $\bar{\mathcal{F}}$ to obtain a finite polyhedral complex of
parallelotopes $\mathcal{S}$. 

We now show that the union of all elements in
$\mathcal{S}$ contains $0$ in its interior. For every maximal element $C$ of
$\bar{\mathcal{F}}$, there exists $\varepsilon_C > 0$ such that $\delta r \in
\Pi_C$ for all $r\in C$ and all $0 \leq \delta \leq \varepsilon_C$, where $\Pi_C$
is the parallelotope in $\mathcal{S}$ corresponding to $C$. Observe that
$\bar{\mathcal{F}}$ is complete because $\mathcal{F}$ is complete and
$\bar{\mathcal{F}}$ is a triangulation of $\mathcal{F}$. Therefore, choosing
$\varepsilon_2 = \min\{\, \varepsilon_C\st C \in \bar{\mathcal{F}}\,\}$, the ball $B_{\varepsilon_2}(0)$ is
contained in the union of all the parallelotopes in $\mathcal{S}$.  

From Lemma~\ref{lemma:simplex_S}, for every parallelotope $\Pi\in \mathcal S$, there exists $g^\Pi\in \R^k$ such that $\pi'(r) = g^\Pi\cdot r$ for all $r \in \Pi$.  Let $K = \max\{\|g^\Pi\|\st \Pi \in \mathcal{S}\}$.  By the Cauchy--Schwarz inequality, $|\pi'(r)| \leq \|g^\Pi\| \|r\| \leq K \|r\|$ for all $r \in \Pi$.  Since the union of all parallelotopes in $\mathcal S$ contains $B_{\varepsilon_2}(0)$ in its interior, 
$| \pi'(r)|\leq K\|r\|$ for all $r\in\R^n$ satisfying $\|r\|<\varepsilon_2$, i.e., $\pi'$ is locally Lipschitz continuous at the origin.

Since $\pi'$ is a subadditive function that is locally Lipschitz continuous at the origin, Lemma~\ref{lem:LLC+SA=LC} shows that $\pi'$ is (globally) Lipschitz continuous.
\end{proof}

The following lemma will be the main tool for using translates of patches to prove that $\pi'$ is affine in the maximal cells of $\P$. 

\begin{lemma}[Finite path of patches]\label{lemma:path}
  Let $P \subseteq \R^k$ be a full-dimensional polyhedron 
  and $\Pi \subseteq \R^k$ be a full-dimensional parallelotope with~$0\in \Pi$.
  Let $x,y$ be points that lie in $\intr(P)$. 
Then there exist a number $0<\varepsilon\leq 
  1$, an integer~$m$, and points $x^0= x,x^1,x^2,\dots,x^m = y\in P$ such that: 
\begin{enumerate}[\rm(i)]
\item $x^j + \varepsilon \Pi \subseteq P$ for $j=0,\dots,m$,
\item $(x^j+\varepsilon \Pi) \cap (x^{j+1}+\varepsilon \Pi)$ is non-empty for $j=0,\dots,m-1$.
\end{enumerate}
\end{lemma}

\begin{proof}
After a linear change of coordinates, we can assume that the parallelotope~$\Pi$ is the unit cube $[0,1]^k$.

Since $x, y\in\intr(P)$, then there exists $\delta > 0$ such that both
  $B_\delta(y)$ and $B_\delta(x)$ lie within~$P$.  Choose $0<\varepsilon \leq 1$ so that $\varepsilon
  \Pi \subset B_\delta(0)$. Therefore, $x + \varepsilon\Pi \subseteq P$ and $y + \varepsilon\Pi \subseteq P$.  Let $m > \|y^0-x^0\|_\infty / \varepsilon$ be an integer.  

 Let 
  $$x^j = x^0 + \frac{j}{m} (y^0 - x^0)\quad\text{for $j=1,\dots,m$;}$$
  thus $x^m = y^0$. 
 Since $x + \varepsilon\Pi \subseteq P$ and $y + \varepsilon\Pi \subseteq P$, by convexity $x^j+\varepsilon \Pi \subseteq P$ for all $j = 0, \ldots, m$. In particular, $x^j \in P$ for all $j = 0, \ldots, m$.
  Moreover $\|x^{j+1} - x^j\|_\infty < \varepsilon\leq 1$, and thus 
  $(x^j+\varepsilon \Pi) \cap (x^{j+1}+\varepsilon \Pi)$ is non-empty.
\end{proof}

\begin{lemma}[$\pi'$ is affine on each maximal cell of~$\pi$]\label{lemma:shifted-affine}
  Let $P_0 \in \P_i$ be a maximal cell containing the origin
  and let $\Pi$ be a full-dimensional parallelotope with $0\in \Pi \subseteq P_0$
  such that $\pi'(x) = g'\cdot x$ for all $x\in \Pi$.
  Let $P$ be a maximal cell in $\mathcal P_i$ and $\bar x \in \intr(P)$. 
  Then $\pi'(x) = g' \cdot (x - \bar x) + \pi'(\bar x)$ for all $x \in P$.
\end{lemma}
\begin{proof} 
  First consider $\bar y \in \intr(P)$.  Let $\varepsilon$ and $x^0 = \bar x,\dots,x^m = \bar y\in P$ be the data from applying 
  Lemma~\ref{lemma:path} on $x,y, P$ and $\Pi$.  Fix any $j\in\{0,\dots,m\}$ and consider an arbitrary $s\in \varepsilon \Pi$. 
Since $P \in \P_i$, $\pi(x^j + s) - \pi(x^j) = \gp^i\cdot s = \pi(s)$, where the second equality follows from $\Pi \subseteq P_0 \in \P_i$. 
Therefore, $\pi(x^j + s) = \pi(x^j) + \pi(s)$ and so the pair $(x^j, s)$ is in $E(\pi) \subseteq E(\pi')$. 
Therefore, $\pi'(x^j + s) = \pi'(x^j) + \pi'(s)$, and thus $\pi'(x^j + s) = \pi'(x^j) + g'\cdot s$. 
Thus $\pi'$, restricted to each $x^j+\varepsilon \Pi$, is an affine function with
  gradient~$g'$, which we write as
  \begin{equation}\label{eq:continuation-j}
    \pi'(x) = g' \cdot (x - \bar x) + \alpha_j 
    \quad\text{for $x\in x^{j}+\varepsilon \Pi$}
  \end{equation}
  for some real number~$\alpha_j$.

  For all $j=0,1,\ldots, m$, we prove that  $\alpha_j = \pi'(\bar x)$ 
  and therefore $\pi'(x) = g'\cdot (x - \bar x) + \pi'(\bar x)$ holds for all $x\in
  x^j+\varepsilon \Pi$.  We do this by induction on~$j$.
  For $j=0$, this holds since $\bar x = x^0$.
  Now let $j+1 > 0$ and assume~$\alpha_{j} = \pi'(\bar x)$. 
  Let $z$ be any point in the intersection $(x^j+\varepsilon \Pi) \cap
  (x^{j+1}+\varepsilon \Pi)$, which is non-empty by Lemma~\ref{lemma:simplex_S}. 
  By evaluating~\eqref{eq:continuation-j} for $j$ and $j+1$ at
  $x=z$, we see that in fact~$\alpha_{j+1}=\pi'(\bar x)$. Therefore, in particular, $\pi'(\bar y) = \pi'(x^m) = a'(\bar y - \bar x) + \pi'(\bar x)$. 
  
  This shows that for every $x \in \intr(P)$, $\pi'(x) = g' \cdot (x - \bar x) + \pi'(\bar x)$. By Lemma~\ref{lemma:theta-lipschitz}, $\pi'$ is continuous, and therefore the equation extends from $\intr(P)$ to all of~$P$.
\end{proof}

\begin{prop}\label{prop:theta-compatible}
  The function~$\pi'$ is a piecewise linear function 
compatible with~$\{\mathcal{P}_i\}_{i = 1}^{k + 1}$.
\end{prop}
\begin{proof}
  Fix $i\in\{1,\dots,k+1\}$. Since $\pi$ satisfies the hypotheses of Lemma~\ref{lem:gi-cone}, there exists a maximal cell $P_0\in\mathcal \P_i$ containing the origin. Since $P_0$ is a full-dimensional polyhedron containing the origin, there exists a full-dimensional parallelotope $\Pi$ with
  $0\in \Pi$ and $\Pi+\Pi \subseteq P_0$. Let $g'$ be the vector from
  Lemma~\ref{lemma:simplex_S} such that $\pi'(r) = g'\cdot r$ for $r\in \Pi$. 
  Define $\gt^i = g'$.
  Now let $P$ be any maximal cell in $\P_i$ and pick any $y\in\relint(P)$.
  By Lemma~\ref{lemma:shifted-affine}, 
  $$\pi'(r) = \gt^i \cdot (r - y) +
  \pi'(y) = \gt^i\cdot r + \delta_P$$ for $r\in P$, 
  where we set $\delta_P = \pi'(y) - \gt^i\cdot y$. 
  Thus $\pi'$ is a piecewise linear function compatible with~$\{\mathcal{P}_i\}_{i = 1}^{k+1}$.
\end{proof}

        Notice that this compatibility implies that there exist vectors $\gt^1, \gt^2, \ldots, \gt^{k + 1}$ corresponding to $\P_1, \ldots, \P_{k+1}$ such that for any $P \in \P_i$, there exists $\delta_P$ such that $\pi'(r) = \gt^i\cdot r + \delta_P$. However, note that we have not shown $\gt^1, \gt^2, \ldots, \gt^{k + 1}$ to be all distinct.


\subsection{Constructing a system of linear equations}\label{sec:gradient-set}

        As the next step in proving that $\pi = \pi'$, we construct a system of linear equations which is satisfied by both $\gp^1, \ldots, \gp^{k+1}$ and $\gt^1, \ldots, \gt^{k+1}$.  
        
The system has two sets of constraints, the first of which follows from
Theorem \ref{thm:minimal} and Lemma \ref{lemma:integration}.  The second set
of constraints is more involved.  Consider two adjacent cells $P, P' \in
\mathcal{P}$ that contain a segment $[x, y] \subseteq \mathbb{R}^k$ in their
intersection. Along the line segment $[x,y]$, the gradients of $P$ and $P'$
projected onto the line spanned by the vector $y - x$ must agree; the second
set of constraints captures this observation. We will identify a set of vectors $r^1, \ldots, r^{k+1}$ such that every subset of $k$ vectors is linearly independent and such that each vector $r^i$ is contained in $k$ cells of $\P$ with different
gradients. We then use the segment $[0,r^i]$ to obtain linear equations involving the gradients of $\pi$ and $\pi'$. The fact that every subset of $k$ vectors is linearly independent will be crucial in ensuring the uniqueness of the system of equations.

\begin{remark}
  In the case $k=2$, in the terminology of~\cite{3slope}, these vectors would all be
  \emph{directions} of the piecewise linear function~$\pi$; see
  also the discussion in Appendix~\ref{sec:generalizes}.
\end{remark}

To show the existence of such a set of vectors, we utilize the following
classical lemma in combinatorial topology.         

\begin{lemma}[KKM \cite{kkm,fpt}]\label{lem:KKM}
        Consider an $n$-simplex $\conv(u^j)_{j = 1}^{n}$. Let $F_1, F_2, \ldots, F_n$ be closed sets such that for all $I \subseteq \{1, \ldots, n\}$, the face $\conv(u^j)_{j \in I}$ is contained in $\bigcup_{j \in I} F_j$. Then the intersection $\bigcap_{j = 1}^n F_j$ is non-empty. 

\end{lemma}        \begin{lemma}\label{lemma:kkmDirections}
                There exist vectors $r^1, r^2, \ldots, r^{k + 1} \in \R^k$ with the following properties:
                \begin{enumerate}[\rm(i)]
                        \item For every $i,j,\ell \in \{1, \ldots, k + 1\}$ with $j, \ell$ different from $i$, the equations $r^i\cdot \gp^{j} = r^i\cdot \gp^{\ell}$ and $r^i\cdot \gt^{j} = r^i\cdot \gt^{\ell}$ hold.
                        
                        \item $\cone(r^i)_{i = 1}^{k + 1} = \R^k$.
                \end{enumerate}
        \end{lemma}

        \begin{proof}We consider the neighborhood $B_\varepsilon(0)$ of the origin given by the local finiteness assumption (see Definition~\ref{assum:fan}). 
        Let $F_i = \bigcup_{P \in \mathcal{P}_i} (P \cap {\bar
          B_\varepsilon(0)})$, namely the set of points in the closed ball~${\bar B_\varepsilon(0)}$ for which $\pi$ has gradient $\gp^i$. 
       Since ${\bar B_\varepsilon(0)}$ is compact, Proposition~\ref{prop:complex-finite-on-compact} says that only finitely many terms are non-empty in the union $\bigcup_{P\in \P_i}(P \cap {\bar B_\varepsilon(0)})$. Moreover each term $P \cap {\bar B_\varepsilon(0)}$ is closed as it is the intersection of a polyhedron with a closed ball. Thus, each $F_i$ is a finite union of closed sets and therefore is closed. Our first goal is to show that, for each $i =1, \ldots, k+1$, there is a vector $r^i$ which belongs to $\bigcap_{j \neq i} F_j$.
        
        In order to better understand how the sets $F_i$ intersect, we start by defining the set $H_i = \{\, r \in \R^k \st \gp^i\cdot r \leq 0\,\}$. The crucial property of this set is that the gradient of $\pi$ at these points must be \emph{different} from $\gp^i$, at least around the origin. 

\begin{claim}\label{claim:Hi_interior}
For every $i = 1, \ldots, k+1$, the set $F_i$ is disjoint with $H_i$.
\end{claim}

\begin{proof}
Suppose to the contrary that there exists $P \in \P_i$ and $r \in H_i \cap P \cap {\bar B_\varepsilon(0)}$. Since $r \in {\bar B_\varepsilon(0)},$ the entire segment $[0,r]$ is contained in
$P$. Moreover, $\gp^i\cdot r \leq 0$ as $r \in H_i$.  
Thus $\pi(\lambda r) = \lambda \gp^i\cdot r\leq 0$ for all $0 \leq \lambda
\leq 1$. Since $\pi$ is piecewise linear with a locally finite cell complex and subadditive, Lemmas~\ref{lem:PWL+LF=LLC} and~\ref{lem:LLC+SA=LC} show that $\pi$ is Lipschitz continuous. Therefore, $\pi$ satisfies the hypotheses of Lemma~\ref{lem:genuinely-k-dim} and we conclude that $\pi$ is not genuinely $k$-dimensional. This is a contradiction.
\end{proof}

\begin{figure}
\label{fig:KKM}
\centering
\vspace{-10ex}
\ifpdf
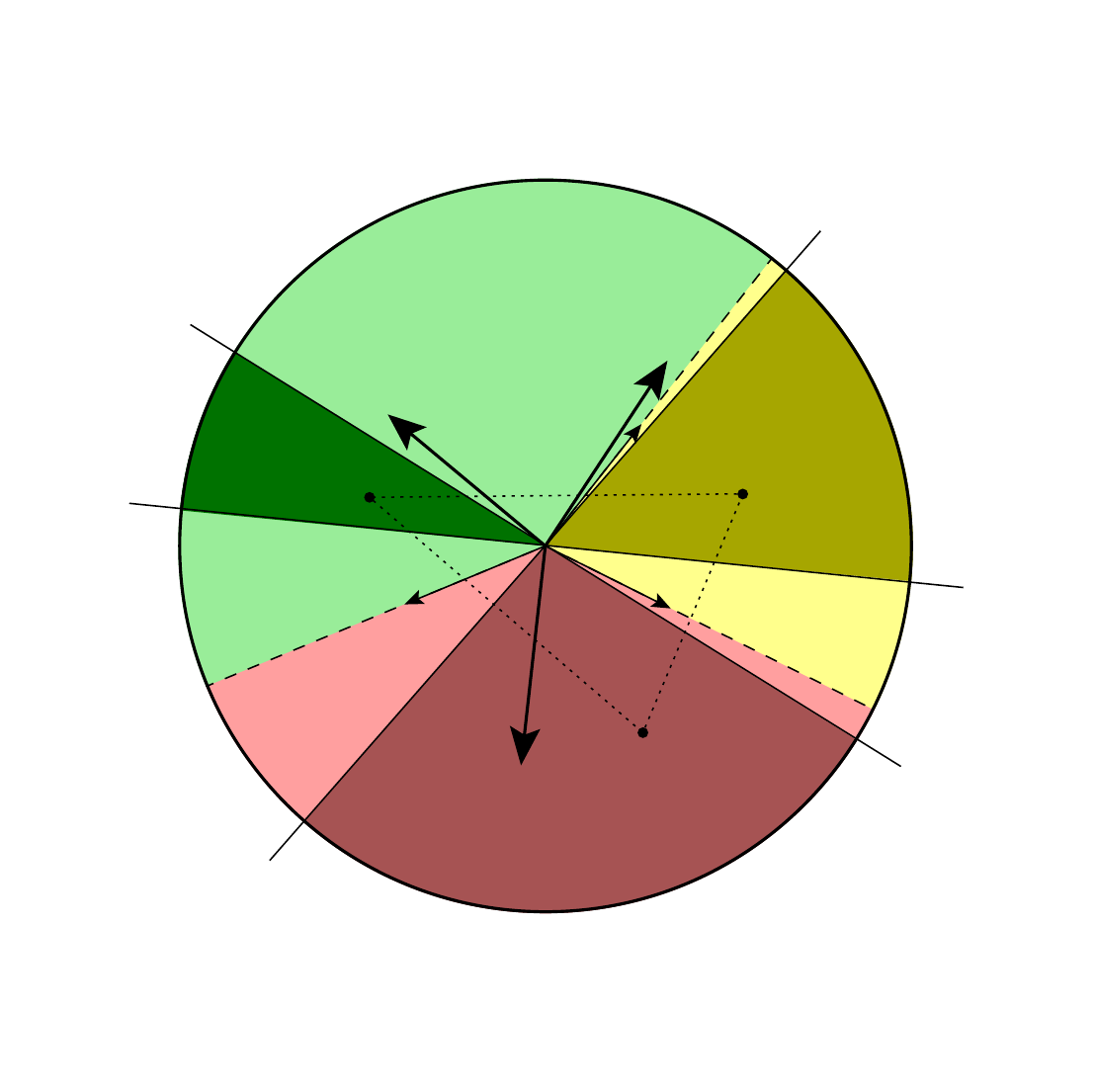
\else
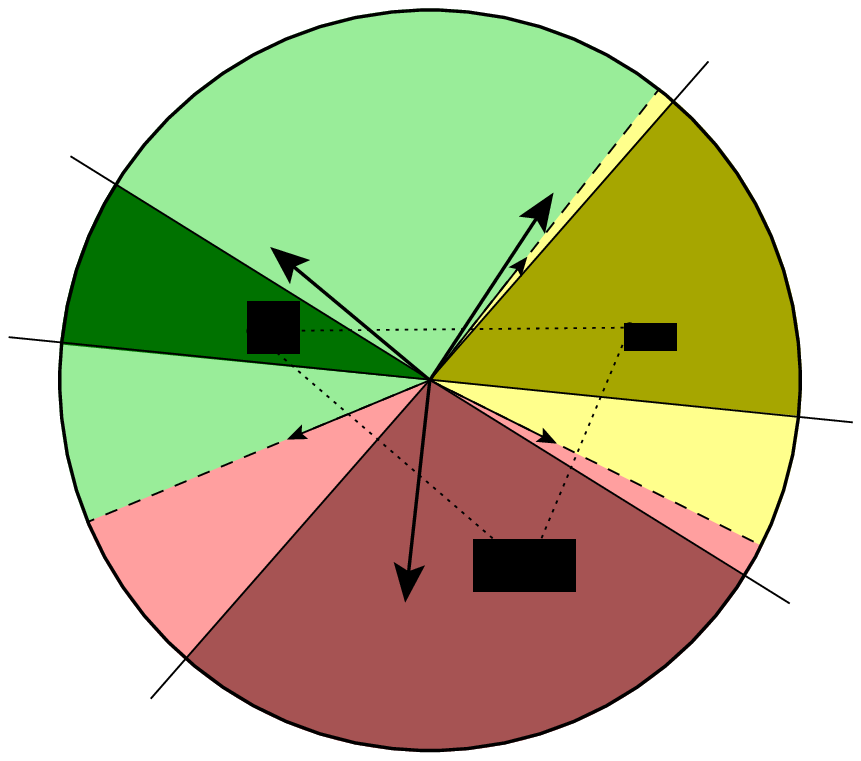
\fi
\vspace*{-5ex}\par
\caption{The geometry of the proof of Lemma~\ref{lemma:kkmDirections}.   
  Each cone $C_i$ (shaded in dark colors) is the intersection of the
  halfspaces~$H_j$ (defined by the gradients~$\bar g^j$) for $j\neq i$. 
  Near the origin (within the ball~$B_\epsilon(0)$), 
  each point of~$C_i$ lies in the set $F_i$ of points
  where the function~$\pi$ has gradient~$\bar g^i$ (shaded in light colors).
  Picking points $v_i$ near the origin in the interior of~$C_i$, 
  we construct a simplex~$\Delta$ with $0$ in its interior. 
  By applying the KKM Lemma to each of its facets~$\Delta_i$, we
  show the existence of the vectors~$r^i$ with the desired properties.
} 
\end{figure}

        For a subset $I \subseteq \{1, \ldots, k+1\}$, define the cone $C_I =
        \bigcap_{i \notin I} H_i$ (for convenience of notation, we use $C_j$
        instead of $C_{\{j\}}$ for a singleton set). From the above claim, for all $i \notin I$ we have $F_i$ disjoint with $C_I$; since $B_\varepsilon(0) \subseteq \bigcup_{i =1,\ldots, k+1} F_i$, we get that $C_I \cap B_\varepsilon(0) \subseteq \bigcup_{i \in I} F_i$. Alternatively, the gradient of $\pi$ in any point in $C_I \cap B_\varepsilon(0)$ must be within the set $\{\gp^i\}_{i \in I}$. We need the following technical property of the cones $C_I$. 

\begin{claim}\label{claim:full-dim}
$C_j$ is full-dimensional for all $j =1, \ldots, k+1$.
\end{claim}

\begin{proof}
Observe that the polar cone $$(C_j)^\circ = \bigl\{\,\textstyle\sum_{i\neq j}\lambda_i \gp^i \st \lambda_i \geq 0\,\bigr\}$$ does not contain any lines because the set $\{\gp^i\}_{i\neq j}$ is linearly independent by Lemma~\ref{lem:lin-indep} and Lemma~\ref{lem:gi-cone}. Hence, $C_j$ is full-dimensional.
\end{proof}

        In order to continue analyzing how the sets $F_i$ intersect, it is useful to focus on a full-dimensional simplex $\conv(v^j)_{j = 1}^{k+1}$ around the origin. 
More precisely, Claim~\ref{claim:full-dim} allows us to pick $v^j \in \intr(C_j) \cap B_\varepsilon(0)$ for every $j = 1, \ldots, k+1$. 
Since $v^j \in \intr(C_j)$, we have $v^j\cdot \gp^i <0$ for all $i \neq j$. 
Then employing Lemma~\ref{lem:cone-span} with $a^i = \gp^i$ and
$b^i = v^i$, we deduce that $\cone(v^i)_{i=1}^{k+1} = \R^k$. Therefore, 
$\Delta = \conv(v^i)_{i=1}^{k+1}$ is indeed a full-dimensional simplex.

        Since $\Delta \subseteq B_\varepsilon(0) \subseteq \bigcup_{i =1,\ldots, k+1} F_i$, the sets $F_i$ form a closed cover of $\Delta$, and in particular they form a closed cover of each facet $\Delta_i = \conv(v^j)_{j \neq i}$. We will show that, for each $i = 1, \ldots, k+1$, there is a point $r^i$ in $\Delta_i$ which belongs to $\bigcap_{i \neq j} F_j$. For that, we apply the KKM Lemma (Lemma~\ref{lem:KKM}) to the simplex $\Delta_i$.

          To do so, we need to show that for every $I \subseteq \{1, \ldots, k+1\} \setminus \{i\}$, the face $\conv(v^j)_{j \in I}$ is contained in $\bigcup_{j \in I} F_j$. 
To see that this holds, take $I \subseteq\{1, \ldots, k+1\} \setminus \{i\}$.
By definition, for every $j \in I$ we have $v^j \in \intr(C_j) \cap B_\varepsilon(0) \subseteq C_I \cap B_\varepsilon(0)$. 
Since $C_I \cap B_\varepsilon(0)$ is convex, it follows that the entire face $\conv(v^j)_{j \in I}$ belongs to $C_I \cap B_\varepsilon(0)$.
 As mentioned previously, $C_I \cap B_\varepsilon(0) \subseteq \bigcup_{j \in I} F_j$ and hence the face $\conv(v^j)_{j \in I}$ is contained in $\bigcup_{j \in I} F_j$.
        
        Therefore, for each $i = 1, \ldots, k+1$,  the KKM Lemma (Lemma \ref{lem:KKM}) implies the existence  of a point $r^i \in \Delta_i$ belonging to $\bigcap_{j \neq i} F_j$ as desired.
        
        Now it is easy to see that  $r^1, \ldots, r^{k+1}$ satisfy
        property~(i) as claimed.
        Fix $i \in \{1, \ldots,
        k+1\}$. Consider $j \neq i$ and let $P \in \P_j$ contain $r^i$; notice
        that actually $r^i \in P \cap \Delta_i \subseteq P \cap
        B_\varepsilon(0)$. Since $\P \cap B_\varepsilon(0) = \mathcal{F} \cap B_\varepsilon(0)$ for some polyhedral fan $\mathcal{F}$, $P$ also contains the entire segment $[0,r]$. Since $\pi$ is affine in $P$ with gradient $\gp^j$, it follows that $\pi(r^i) - \pi(0) = r^i \cdot \gp^j$; this implies that for all $j, \ell \neq i$ we have $r^i \cdot \gp^j = r^i \cdot \gp^{\ell}$. Similarly, since $\pi'$ is a piecewise linear function compatible with $\{\P_i\}_{i=1}^{k+1}$, again we have that $\pi'(r^i) - \pi'(0) = r^i \cdot \gt^j$ for all $j \neq i$, and hence $r^i \cdot \gt^j = r^i \cdot \gt^{\ell}$ for all $j, \ell \neq i$. 

        Finally, we prove that  $r^1, \ldots, r^{k+1}$ satisfy property~(ii)
        as claimed.  Because $r^i \in \bigcap_{j \neq i} F_j$, Claim~\ref{claim:Hi_interior} directly implies that $r^i \notin H_j$ for every $j \neq i$, namely $r^i\cdot \gp^j > 0$ when $j \neq i$. Now using Lemma~\ref{lem:cone-span} with $a^i = -\gp^i$ and $b^i = r^i$, we deduce that $\cone(r^i)_{i=1}^{k+1} = \R^k$. This concludes the proof of Lemma~\ref{lemma:kkmDirections}.
\end{proof}

        We finally present the system of linear equations that we consider.
        
        \begin{corollary} \label{lemma:linear-system}
                Consider vectors $a^1, a^2, \ldots, a^{k + 1} \in \mathbb{Z}^k - f$ such that $\cone(a^i)_{i=1}^{k+1} = \R^k$. Also, let $r^1, r^2, \ldots, r^{k + 1}$ be the vectors given by Lemma \ref{lemma:kkmDirections}. Then there exist $\mu_{ij} \in \R_+$, $i, j \in \{1, \ldots, k+1\}$ with $\sum_{j = 1}^{k + 1} \mu_{ij} = 1$ for all $i \in \{1, \ldots, k+1\}$ such that both $\gt^1, \ldots, \gt^{k+1}$ and $\gp^1, \ldots, \gp^{k+1}$ are solutions to the linear system 
        \begin{equation}\label{eq:linear-system}
\begin{aligned}
\textstyle\sum_{j=1}^{k+1} (\mu_{ij}a^i)\cdot \gs^j & = 1 &\qquad& \text{for all }  i \in \{1, \ldots, k+1\}, \\
r^i\cdot \gs^{j} - r^i\cdot \gs^{\ell} & = 0 & \qquad& \text{for all } i,j, \ell \in \{1, \ldots, k+1\} \textrm{ such that } i \neq j, \ell, 
\end{aligned}
\end{equation}
with variables  $g^1, \ldots, g^{k+1}\in \R^k$.
        \end{corollary}
        
        \begin{proof}
                Feasibility for the first set of constraints follows directly from the minimality of $\pi$ and $\pi'$, Theorem~\ref{thm:minimal} and Lemma~\ref{lemma:integration}. Feasibility for the second set of constraints follows from Lemma \ref{lemma:kkmDirections} (i). 
        \end{proof}
        
        We remark that we can always find vectors $a^1, a^2, \ldots, a^{k + 1} \in \mathbb{Z}^k - f$ such that $\cone(a^i)_{i=1}^{k+1} = \R^k$, so the system above indeed exists. 


\subsection{Unique solution of the linear system}\label{sec:unique}

We now analyze the solution set of~\eqref{eq:linear-system}, which will be rewritten as a system of $k(k+1)$ linear equations for the $k+1$ gradient vectors, i.e., in $k(k+1)$ variables. We will show the gradients of $\pi$\ and $\pi'$\ coincide by demonstrating  that this system either has  no solutions or has a unique solution.
Recall from linear algebra that, given a square matrix $A$ and a vector $b$, if the augmented matrix $[b\; A]$ has full row rank, then the linear system $Ay = b$ either has no solutions or has a unique solution.

\begin{prop}\label{prop:gradients-equal}
$\gp^i = \gt^i$ for every $i=1, \ldots, k+1$.
\end{prop}
\begin{proof}
We wish to show that the system~\eqref{eq:linear-system} either has no solution or a unique solution.
We begin by rewriting the system in terms of some new variables. Since for any fixed~$i$, the value of  $r^i g^j$ must coincide for all $j=1, \ldots, k+1$, $i\neq j$, we reformulate the system~\eqref{eq:linear-system} by introducing $z \in \R^{k+1}$ such that $z_i$ is this value. We can rewrite the system \eqref{eq:linear-system} as 
\begin{equation}\label{eq:linear-system2}
\begin{aligned}
\textstyle\sum_{j=1}^{k+1} (\mu_{ij}a^i)\cdot \gs^j & = 1 &\qquad& \text{for all } i = 1, \ldots, k+1 \\
r^i\cdot \gs^{j} - z_i & = 0 & \qquad& \text{for all } i,j = 1, \ldots, k+1, \text{ such that } i \neq j.
\end{aligned}
\end{equation}

Note that there is a one-to-one mapping between solutions of \eqref{eq:linear-system} and \eqref{eq:linear-system2}. We now rearrange the variables and the constraints of \eqref{eq:linear-system2} so that it can be represented as $Ay = b$, where
$$
A = \left[ \begin{array}{cc} \Delta & O_{k+1 \times k+1}\\ R & I' \end{array}  \right],\quad y = [\gs^1, \ldots, \gs^{k+1}, z]^T, \quad \text{and} \quad b = [1, \ldots, 1, O_{1\times k} \ldots, O_{1\times k}]^T,
$$
where $\Delta$ is a $(k+1) \times k(k+1)$-matrix, $R$ is a $k(k+1) \times k(k+1)$-matrix, $I'$ is a $k(k+1) \times (k+1)$-matrix and $O_{i \times j}$ is the $i\times j$-matrix with all zero entries.
The $i$-th row of $\Delta$, $i=1, \ldots, k+1$, is given by $(\mu_{i1}a^i, \ldots, \mu_{i(k+1)}a^i)$ where $a^i$ is written as a row vector.
The matrix $R$ has a block diagonal structure: 
$$
R = \left[ \begin{array}{ccc} R_1 &&\\& \ddots & \\ && R_{k+1}          \end{array}\right],
$$ 
where each $R_i$ is a $k\times k$-matrix. For each $i=1, \ldots, k+1$, the matrix $R_i$ has rows $r^j$, $j \neq i$.

The matrix $I'$ has entries corresponding to the coefficients on $z$, and will be written as
$$
I' = \left[ \begin{array}{c} -I_1 \\ \vdots \\ -I_{k+1}\end{array}\right],
$$
where $I_i$ is a $k\times (k+1)$-matrix obtained from the $k\times k$ identity matrix with the $0$ column inserted as the $i$-th column.

We now argue that the matrix $[b \; A]$ has full row rank. Since $A$ is a $(k+1)^2 \times (k+1)^2$ square matrix, if $[b\;  A]$ has full row rank, the system $Ay = b$ either has a unique solution or no solution. 

We use one further trick to prove $[b \; A]$ has full row rank: we analyze the row rank of the matrix $[b\;  D \;A]$, where $D$ is a $(k+1)^2 \times k$-matrix of all zero entries. The rank of $[b \; D \; A]$ is the same as $[b \; A]$ and we now show that $[b \; D \; A]$ has full row rank. We now perform the  block row and column operations on the matrix

\begin{align*}
[\,b\mid D \mid A\,] 
=& 
{\left[ \begin{array}{c|c|cccc} 1 & O_{1\times k} &  \mu_{11}a^1 & \ldots &  \mu_{1(k+1)}a^1  & O_{1\times k} \\
                                                \vdots & \vdots &  \vdots && \vdots & \vdots \\
                                                1 & O_{1\times k} &  \mu_{(k+1)1}a^{k+1} & \ldots &  \mu_{(k+1)(k+1)}a^{k+1}  & O_{1\times k} \\ \hline
                                                O_{k\times 1} & O_{k\times k} & R_1 & & & -I_1 \\ 
                                                \vdots & \vdots & & \ddots & & \vdots \\
                                                O_{k\times 1} & O_{k\times k} &&&R_{k+1} & -I_{k+1}
                                                \end{array}\right]}.\\
\intertext{First, add all the block columns of $A$ corresponding to each $g^1, \ldots, g^{k+1}$ to
the block $D$, giving (recall that $\sum_{j=1}^{k+1}\mu_{ij} = 1$ for all $i \in \{1, \ldots, k+1\}$)}
&               {\left[ \begin{array}{c|c|cccc} 1 & a^1 &  \mu_{11}a^1 & \ldots &  \mu_{1(k+1)}a^1  & O_{1\times k} \\
                                                \vdots & \vdots &  \vdots && \vdots & \vdots \\
                                                1 & a^{k+1} &  \mu_{(k+1)1}a^{k+1} & \ldots &  \mu_{(k+1)(k+1)}a^{k+1}  & O_{1\times k}\\ \hline
                                                O_{k\times 1} & R_1 & R_1 & & & -I_1 \\ 
                                                \vdots & \vdots & & \ddots & & \vdots \\
                                                O_{k\times 1} & R_{k+1} &&&R_{k+1} & -I_{k+1}
                                                \end{array}\right]}.\\
\intertext{%
Second, in the last matrix above, multiply the last block of $k+1$ columns (corresponding to the variables~$z_i$)  on the right by the matrix $\bar R$, which is the $(k+1)\times k$ matrix whose rows are the $k+1$ vectors $r^i$, for all $i=1, \ldots, k+1$. Note that $-I_i\bar R = -R_i$ for every $i = 1, \ldots, k+1$. Hence, if we multiply the last block of columns with $\bar R$ and add to the second block of columns in the last matrix above, we obtain}
&               {\left[ \begin{array}{cc|ccc|c} 1 & a^1 &  \mu_{11}a^1 & \ldots &  \mu_{1(k+1)}a^1  & O_{1\times k} \\
                                                \vdots & \vdots &  \vdots && \vdots & \vdots \\
                                                1 & a^{k+1} &  \mu_{(k+1)1}a^{k+1} & \ldots &  \mu_{(k+1)(k+1)}a^{k+1}  & O_{1\times k}\\ \hline
                                                O_{k\times 1} & O_{k\times k} & R_1 & & & -I_1 \\ 
                                                \vdots & \vdots & & \ddots & & \vdots \\
                                                O_{k\times 1} & O_{k\times k} &&&R_{k+1} & -I_{k+1}
                                                \end{array}\right]}.
\end{align*}
 The final matrix has an upper triangular block structure. The blocks on the diagonal are 
$$A' = \left[\begin{array}{rc}1&  a^1 \\ \vdots &\vdots \\ 1& a^{k+1}\end{array}\right] \text{ and } R_1, \ldots, R_{k+1}.$$
Each $R_i$ has full row rank since every proper subset of $\{r^1, \ldots, r^{k+1}\}$ is linearly independent by Lemma~\ref{lem:lin-indep}. Also, $A'$ has full row rank because $a^1, \ldots, a^{k+1}$ are affinely independent since $\cone(a^i)_{i=1}^{k+1} = \R^k$. Hence, we have shown that $[b\; D\; A]$ has full row rank.

Therefore the system \eqref{eq:linear-system} has either no solutions or has a unique solution.  Since Corollary~\ref{lemma:linear-system} shows that both $\gt^1, \ldots, \gt^{k+1}$ and $\gp^1, \ldots, \gp^{k+1}$ are solutions to~\eqref{eq:linear-system}, it follows that  $\gt^i = \gp^i$ for all $i = 1, \ldots, k+1$.
\end{proof}

\subsection{Conclusion of the proof} \label{sec:conclusion}

 Since both $\pi$ and $\pi'$ are minimal, Theorem \ref{thm:minimal} guarantees that $\pi(0) = \pi'(0) = 0$. Proposition~\ref{prop:gradients-equal} shows that $\gp^i = \gt^i$ for all $i=1, \ldots, k+1$.  From Lemma \ref{lemma:integration}, for every $r \in \R^k$ there exist $\mu_1, \mu_2, \ldots, \mu_{k+1}$ such that $$\pi(r) = \pi(0) + \sum_{i=1}^{k + 1} \mu_i (\gp^i \cdot r) = \pi'(0) + \sum_{i=1}^k \mu_i (\gt^i\cdot r) = \pi'(r).$$ This proves that $\pi = \pi'$ and concludes the proof of Theorem \ref{thm:main}. 


\clearpage
\appendix

\section{Appendix}
\subsection{Proof of the existence of minimal valid functions -- Theorem \ref{thm:minimal1}}\label{sec:minimality-proof}
\begin{proof}[Proof of Theorem \ref{thm:minimal1}]
Consider the non-empty set $\Sigma$ of valid functions $\pi'$ with $\pi' \leq \pi$ (the set is non-empty because $\pi \in \Sigma$). We now consider $(\Sigma, \leq)$ as a partially ordered set, where the partial order is imposed by the relation $\pi_1 \leq \pi_2$ for $\pi_1, \pi_2 \in \Sigma$. If can show that every chain in $(\Sigma, \leq)$ has a lower bound, then applying Zorn's lemma we would conclude that $\Sigma$ has a minimal element which will be the minimal function $\pi'$ we are looking for.

Consider any chain $\mathcal{C}$ in $(\Sigma, \leq)$, i.e., for $\pi_1, \pi_2 \in \mathcal{C}$ either $\pi_1 \leq \pi_2$ or $\pi_2 \leq \pi_1$. 
Consider the function $\pi_{\mathcal{C}}$ defined as follows: $\pi_{\mathcal{C}}(r) = \inf_{\pi' \in \mathcal{C}}\pi'(r)$. 
We claim that $\pi_{\mathcal{C}} \in \Sigma$. 
We only need to verify that it is a valid function; it is clear that $\pi_{\mathcal{C}} \leq \pi$. 
Since $\pi' \geq 0$ for all $\pi' \in \mathcal{C}$, $\pi_\mathcal{C} \geq 0$. 

Suppose to the contrary that there exists $s \geq 0$ with finite support such that $f + \sum_{r\in \R^k}rs_r \in \Z^k$, but $\sum_{r\in \R^k}\pi_{\mathcal{C}}(r)s_r < 1$. Let $\{r^1, \ldots, r^n\}$ be the finite support of $s$, i.e., $s_r = 0$ for all $r \not\in\{r^1, \ldots, r^n\}$. 
Let $S = \max\{s_{r^1}, \ldots, s_{r^n}\}$ and let $\varepsilon = 1 - \sum_{r\in \R^k}\pi_{\mathcal{C}}(r)s_r > 0$. Since $\pi_{\mathcal{C}}(r^i) = \inf_{\pi' \in \mathcal{C}}\pi'(r^i)$, there exists $\pi_i \in \mathcal{C}, i= 1, \ldots, n$ such that $\pi_i(r^i) \leq \pi_{\mathcal{C}}(r^i) + \frac{\varepsilon}{2nS}$. Since $\mathcal{C}$ is a chain, there exists $i^* \in \{1, \ldots, n\}$ such that $\pi_{i^*} \leq \pi_i$ for all $i \in \{1, \ldots, n\}$. Hence, $\pi_{i^*}(r^i) \leq \pi_{\mathcal{C}}(r^i) + \frac{\varepsilon}{2nS}$ for every $i \in \{1, \ldots, n\}$. But then $$\sum_{r\in \R^k}\pi_{i^*}(r)s_r \leq \sum_{r\in \R^k}\pi_{\mathcal{C}}(r)s_r + \sum_{i=1}^n\frac{\varepsilon}{2nS}s_{r^i} \leq 1 - \varepsilon + \frac{\varepsilon}{2nS}nS < 1,$$ which shows that $\pi_{i^*}$ is not a valid function, which is a contradiction because $\pi_{i^*}\in \Sigma$.
\end{proof}


\subsection{Facet Theorem} \label{sec:facet}

        The next lemma shows that a weaker condition than that in the definition of a facet is enough to guarantee facetness. 

\begin{lemma}\label{lemma:minimalsuffices}
 Let $\pi$ be minimal valid function. Suppose that for every minimal valid function~$\pi^*$, we have that $S(\pi) \subseteq S(\pi^*)$ implies $\pi^* = \pi$. Then $\pi$ is a facet.
\end{lemma}
        
        \begin{proof}
                Consider any valid function $\pi'$ (not necessarily minimal) such that $S(\pi) \subseteq S(\pi')$; we show that $\pi' = \pi$.
                
                Suppose to the contrary that there exists $r_1 \in \R^k$ such that $\pi(r_1) \neq \pi'(r_1)$. We claim that actually there is $r_2$ such $\pi(r_2) > \pi'(r_2)$. To see this, first notice that the symmetry condition of $\pi$ (via Theorem \ref{thm:minimal}) guarantees that $\pi(r_1) + \pi(-f -r_1) = 1$. Moreover, it is clear that the solution $\bar{s}$ given by $\bar{s}_{r_1} = \bar{s}_{-f - r_1} = 1$ and $\bar{s}_r = 0$ otherwise is feasible; together, these observations imply that $\bar{s} \in S(\pi)$. Since $S(\pi) \subseteq S(\pi')$, we have that $\bar{s} \in S(\pi')$ and hence $$\pi'(r_1) + \pi'(-f - r_1) = \sum_{r \in \R^k} \pi'(r) \bar{s}_r = 1 = \pi(r_1) + \pi(-f -r_1).$$ Since $\pi(r_1) \neq \pi'(r_1)$, it follows that either $\pi(r_1) > \pi'(r_1)$ or $\pi(-f -r_1) > \pi'(-f - r_1)$, and the claim holds.  
                
                Now consider a minimal valid function $\pi^* \leq \pi'$ (which exists by Theorem \ref{thm:minimal1}). Notice that $S(\pi') \subseteq S(\pi^*)$: for $\bar{s} \in S(\pi')$, using its validity we get $1 \le \sum_{r \in \R^k} \pi^*(r) \bar{s}_r \le \sum_{r \in \R^k} \pi'(r) \bar{s}_r = 1$, hence equality hold throughout and $\bar{s} \in S(\pi^*)$. Since $S(\pi) \subseteq S(\pi')$, we get that $S(\pi) \subseteq S(\pi^*)$. However, $\pi \neq \pi^*$, since there is $r_2$ such that $\pi(r_2) > \pi'(r_2) \ge \pi^*(r_2)$. This contradicts the assumptions on $\pi$, which concludes the proof. 
        \end{proof}

        \begin{proof}[Proof of Theorem \ref{thm:facet-theorem}]
                By Lemma \ref{lemma:minimalsuffices}, all we need to show is that for every minimal valid function $\pi'$, $S(\pi) \subseteq S(\pi')$ implies $\pi' = \pi$. We simply show that for every minimal valid function $\pi'$, $S(\pi) \subseteq S(\pi')$ implies $E(\pi) \subseteq E(\pi')$.
                
                So let $\pi'$ be a minimal valid function with $S(\pi) \subseteq S(\pi')$. Consider any $(r_1, r_2) \in E(\pi)$, namely such that $\pi(r_1) + \pi(r_2) = \pi(r_1 + r_2)$. Notice that the solution $\bar{s}$ given by $\bar{s}_{r_1} = \bar{s}_{r_2} = \bar{s}_{-f-r_1-r_2} = 1$ and $\bar{s}_r = 0$ is feasible. Moreover, using symmetry condition of $\pi$ we get that $\bar{s} \in S(\pi)$. Indeed, $$\sum_{r \in \R^k} \pi(r) \bar{s}_r = \pi({r_1}) + \pi({r_2}) + \pi({-f -(r_1 + r_2)}) = \pi({r_1 + r_2}) + \pi({-f -(r_1 + r_2)}) = 1.$$ Since $S(\pi) \subseteq S(\pi')$, the solution $\bar{s}$ also belongs to $S(\pi')$, and now the symmetry condition of $\pi'$ gives $$1 = \sum_{r \in \R^k} \pi'(r) \bar{s}_r = \pi'({r_1}) + \pi'({r_2}) + \pi'({-f - r_1 - r_2}) = \pi'({r_1}) + \pi'({r_2}) + (1 - \pi'({r_1 + r_2})).$$ Thus, $\pi'({r_1}) + \pi'({r_2}) = \pi'({r_1 + r_2})$ and $(r_1, r_2) \in E(\pi')$. This concludes the proof. 
                \end{proof}


\subsection{Proof that Theorem \ref{thm:main} generalizes Theorem 3 of \cite{3slope}} \label{sec:generalizes}

We restate Theorem 3 of \cite{3slope} here using our terminology.
A \emph{direction} of a piecewise linear function $\pi$ with cell complex $\P$
is a linear space parallel to a one-dimensional element of~$\P$, such as an
edge. 

        \begin{theorem}[Theorem 3 of \cite{3slope}] Let $\pi\colon \R^2 \rightarrow \R$ be a minimal valid function. If $\pi$ is piecewise linear with a locally finite cell complex, has 3 slopes and has 3 directions, then $\pi$ is extreme. 
        \end{theorem}
        
        Consider $\pi$ satisfying the hypothesis of the above theorem; it suffices to show that $\pi$ satisfies the hypothesis of Theorem \ref{thm:main}. So let $\{a^i\}_{i =1}^3$ be the gradient set of $\pi$. Lemma 3.6 of \cite{3slope} implies that $\cone(a^i)_{i=1}^3 = \R^3$. The next lemma, which provides a partial converse to Lemma \ref{lem:gi-cone}, shows that this property guarantees that $\pi$ is genuinely 3-dimensional; this implies that $\pi$ satisfies the hypothesis of Theorem~\ref{thm:main} and concludes the proof. 

        \begin{lemma}
                Let $\theta \colon \R^k \rightarrow \R$ be a piecewise linear function with gradient set $\{a^i\}_{i \in I}$. If $\cone(a^i)_{i \in I} = \R^k$, then $\theta$ is genuinely $k$-dimensional. 
        \end{lemma}
        
        \begin{proof}
By means of contradiction, suppose that $\theta$ is not genuinely $k$-dimensional. 
So consider a function $\varphi\colon \R^{k-1} \rightarrow \R$ and a linear map $T \colon \R^k \rightarrow \R^{k - 1}$ such that $\theta = \varphi \circ T$. 
Notice that the kernel of $T$ contains some non-zero vector, and let $v$ be one such vector. 
                
                Let $\theta$ be a piecewise linear function with cell complex $\P$ and take a maximal cell $P \in \P$; we claim that $a^P \cdot v = 0$. Since $P$ is full-dimensional, we can find $x,y \in P$ such that $y - x = \lambda v$ for some $\lambda \neq 0$. Since $T(v) = 0$, we have $\pi(y) = \pi(x + \lambda v) = \varphi(T(x + \lambda v)) = \varphi(T(x)) = \pi(x)$. Moreover, by definition of $\P$, we have that $\pi(r) = a^P \cdot r + \delta_P$ for all $r \in P$. Putting the two previous observations together, we get that $0 = \pi(y) - \pi(x) = a^P \cdot (y - x) = \lambda a^P \cdot v$. Since $\lambda \neq 0$, this implies that $a^P \cdot v = 0$.
                
                However, since this holds for every $P \in \P$, it is clear
                that $\cone(a^i)_{i \in I}$ belongs to the orthogonal
                complement of $v$, and hence does not equal $\R^k$. This
                contradicts the assumption on the vectors $a^i$ and concludes
                the proof of the lemma.
        \end{proof}


\clearpage
\bibliographystyle{plain}
\bibliography{MLFCB_bib}

\end{document}